\documentclass[12pt,leqno]{article}
 \title{On second case of Strong Fermat's Last Theorem  conjecture  } 
\author{Roland Qu\^eme}
\usepackage{amsmath,amsthm,amstext,amsbsy,eufrak}
\newtheorem{thm}{Theorem}[section]
\newtheorem{cor}[thm]{Corollary}

\newtheorem{lem}[thm]{Lemma}
\newtheorem{conj}{Conjecture}
\newtheorem{defi}{Definition}
\newtheorem{notas}{Notations}
\newtheorem{rema}{Remark}
\font\mathbb=msbm10

\newcommand{\F}{\mbox{\mathbb F}}

\newcommand{\Q}{\mbox{\mathbb Q}}
\newcommand{\Z}{\mbox{\mathbb Z}}
\newcommand{\be}{\begin{equation}}
\newcommand{\ee}{\end{equation}}
\newcommand{\bd}{\begin{displaymath}}
\newcommand{\ed}{\end{displaymath}}
\newcommand{\bn}{\begin{enumerate}}
\newcommand{\en}{\end{enumerate}}

\newcommand{\ri}{\rightarrow}

\newcommand{\mk}{\mathfrak}
\newcommand{\ml}{\mathcal}
\newcommand{\mf}{\mathbf}

\newcommand{\ov}{\bar}
\date{2013 Mai 28}
\begin{document}
\maketitle
%
%
%

\tableofcontents
\maketitle

\begin{abstract}
This article  deals with  a  conjecture, introduced in [GQ] (hereinafter $SFLT2$),  which  generalizes  the second case of 
Fermat's Last Theorem:
{\it Let $p>3$ be a prime.
The diophantine  equation 
$\frac{u^p+v^p}{u+v}=w_1^p$
with  $u,v,u+v, w_1\in\Z\backslash\{0\}$, $u,v$  coprime and   $v\equiv 0 \mod p$ 
has no solution.}
Let $\zeta$ be a $p$th primitive root of unity and $K:=\Q(\zeta)$.
A prime $q$ is said {\it $p$-principal} if the class   of any prime ideal $\mk q_K$ of $K$ over $q$ is a 
$p$-power of a class.

Assume  that $SFLT2$ fails for $(p,u,v)$.
Let  $q$ be any  odd prime coprime with $puv$, $f$ the order of $q\mod p$,   $n$ the order of 
$\frac{v}{u}\mod q$, $\xi$ a primitive
$n$th
root of unity, $\mk q$ the prime ideal  $(q,u\xi-v)$ of $\Q(\xi)$.
In this   complement  of the article [GQ]  
revisiting some works of Vandiver, we prove  that,
if  $q$ is {\it $p$-principal} and $n\not=2p$  then   
$$\Big(\frac{1+\xi\zeta^k}{1+\xi\zeta}\Big)^{(q^f-1)/p}\equiv 1\mod \mk q\mbox{\ for\ }k=1,\dots,p-1.$$
We shall  derive ,  by example,  of this congruence that,  for $p$ sufficiently large,   a very large number of primes   should divide $v$. In an other hand we shall show   
that if $q$  is any  prime of order $f\mod p$  dividing $(u^p+v^p)$ then $$(1-\zeta)^{(q^f-1)/p}\equiv p^{-(q^f-1)/p}\mod q, $$
and a result of same nature if $q$ divides $u^p-v^p$,
which  reinforces strongly the  first and second theorem of Furtw\"angler.
The principle of proof relies on the $p$-Hilbert class field theory.
\end{abstract}


 \paragraph{ MSC classification codes:}  11D41; 11R18; 11R37

\maketitle
\section{  Introduction } \label{s1012111}
Let  $p>3$ be a prime, $\zeta:=e^{\frac{2\pi i}{p}}$, $K:=\Q(\zeta)$
the  $p$th cyclotomic number field, $\Z_K$ the ring of integers  of $K$, 
 and  $\mk p=(1-\zeta)\Z_K$  the prime ideal of $\Z_K$ over $p$.
In [Gr2, Conj.\,1.5],  G. Gras  has given  a conjecture  which
implies Fermat's Last Theorem (FLT): we recall here this conjecture which
will be called {\it Strong Fermat's Last Theorem conjecture }, denoted  briefly $SFLT$.

\begin{conj}\label{cj1}  Let $p$ be an odd prime,  set $K = \Q(\zeta)$ 
and ${\mathfrak p} = (\zeta-1)\,\Z[\zeta]$. Then the equation 
    $$(u+v \,\zeta)\,\Z[\zeta] = {\mathfrak p}^\delta\, {\mathfrak w}_1^p$$
    in coprime integers $u,\,v$, where  $\delta$ is any integer $\geq 0$ 
    and  ${\mathfrak w}_1$ is any integral ideal of $K$, 
    has no solution for $p>3$ except the trivial  ones for which
   $u+v \,\zeta = \pm 1$, $\pm \zeta$, $\pm (1+\zeta)$, 
   or  $\pm (1-\zeta)$.
   \end{conj}

The cases $uv(u+v)\not\equiv 0\mod p, \ uv\equiv 0\mod p$, and $u+v\equiv
0\mod p$
are called respectively the {\it  first, second, and special case of $SFLT$}.

\smallskip From some works of Furtw\"angler and Vandiver,  
Gras and  I [GQ] have put
the basis for  a new cyclotomic approach to  Fermat's Last Theorem  by 
introducing some auxiliary  fields of the form $\Q(\mu_{q-1})$ 
with prime  $q\not= p$   in   study  of $SFLT$ equation. 

\smallskip 

In this   article,  we  examine    some particularities of   the  second
case of $SFLT$
(hereinafter $SFLT2$). 
Without loss of generality,  we   choose the
following  formulation of   $SFLT2$ in  all  the sequel:

{\it Let $p>3$ be a prime. The diophantine  equation 
$\frac{u^p+v^p}{u+v}=w_1^p$
with  $u,v,u+v, w_1\in\Z\backslash\{0\}$, $u,v$  coprime and   $v\equiv 0 \mod p$ 
has no solution.}

Observe that we assume $p>3$ because, for $p=3$, the diophantine equation $\frac{u^3+v^3}{u+v} =w_1^3$
has infinitely many solutions $(u,v)$  where 
$$(u,v)=(s^3+t^3-3st^2, 3s^2t-3st^2)= (s+tj)^3,$$
where $s, t$ spans all $s,t\in\Z,\  s+t\equiv 0\mod 3, \ \gcd(s,t)=1$ (see [GQ] remark 2.6).

Thus it is always assumed in the sequel, without further mention, that $p>3$ and $p|v$.
Observe that $SFLT2$ implies the second  case  $FLT2$ of $FLT$.

In the   first subsection, we will fix some general notations, conventions,   definitions, and  in the  second subsection  p. \pageref{s1111232}, we will state  the  main  results of the article.
\subsection{General  notations and definitions } \label{definitions}

\begin{notas}\label{n1}
$ $
{\rm

- Let $g:={\rm Gal}(K/\Q)$, for $k \not\equiv 0 \mod p$ and  $s_k :
\zeta \rightarrow\zeta^k$   
 the $p-1$ distinct elements of $g$. 

- Let $C\ell_K$, $C\ell$  and $C\ell_{[p]}$ be respectively the class group of $K$, 
the $p$-class group of $K$  and  the $p$-elementary
class group of $K$.
For any ideal  $\mk a$  of $K$, 
let us note $c\ell_K(\mk a)$ and $c\ell(\mk a)$ the class of $\mk a$ in $C\ell_K$ and  $C\ell$.

- For any integer  $m>0$,  let $\Phi_m(X)$ be 
the $m$th cyclotomic polynomial and   $\phi(m)$ the Euler indicator. 
For   $a,b\in\Z\backslash \{0\}$, let us define  
$\Phi_{m}(a,b) := b^{\phi(m)}\Phi_{m}\big(\frac{a}{b}\big).$
Clearly $\Phi_m(a,b)\in\Z[a,b]$.
}
\end{notas}

\begin{defi}$ $
A number $\alpha\in K^\times$ prime to $p$, such that $\alpha\Z_K$ is
the $p$th power of an ideal,  is called a pseudo-unit.
The pseudo-unit $\alpha$ is   $p$-primary  (i.e. 
the extension $K(\sqrt[p]{\alpha})/K$ is unramified at $\mk p$) 
if and only if $\alpha$ is congruent 
to a $p$-power $\mod \mk p^p$, see [Gr2] lem 2.1.
\end{defi}

\begin{defi}
A prime $q$ is said {\it  $p$-principal}
if the class $c\ell_K (\mk q_K)\in C\ell_K$  of any prime ideal 
$\mk q_K$ of $\Z_K$ above  $q$ is the $p$th power of a class, which is 
equivalent to $\mk q_K = {\mk a}^p (\alpha)$, for an ideal ${\mk a}$
of $K$ and an $\alpha \in K^\times$.
This contains the case where  the class $c\ell_K (\mk q_K)$ is of order
coprime with  $p$.
\footnote{  Usually $q$ is said $p$-principal if the class $c\ell_K (\mk q_K)$ is of order
coprime with  $p$. We have adopted this generalization of the definition because  we are  interested in this article 
to some  $p$-power  residue symbols $\big(\frac{\eta}{\mk q_K}\big)_{\!\! K}$ for $\eta$ a $p$-primary unit of $K$.
Class field theory implies that $\big(\frac{\eta}{\mk q_K}\big)_{\!\! K}=1$ even with this generalization.}
\end{defi}

We assume  that  $SFLT2$ fails for  $(p,u,v)$;  then $\gamma := u+ v\zeta\in\Z_K$
is a $p$-primary pseudo-unit of $\Z_K$ since $\gamma \equiv u \mod p$
(a generalization of a result of Kummer given  again in [Gr2], Theo.\,2.2).

\medskip

\begin{notas}

Let  $q$ be a prime number dividing $\Phi_n(u,v)$ with $q \not|\ n$   and $n=dp^r$  
where $d$ is prime to $p$ and $r\geq 0$, which implies that 
 $q\not|\  uv$ and $\frac{v}{u}$ is of order $n\mod q$ (see [GQ], Lem 2.11). We have
$$\Phi_{d\,p^r}(u,v) := \prod\  (u\,\psi^i\,\zeta_r^j - v)
\mbox{\ \ for all\ \ } i\ \in (\Z/d\Z)^\times \mbox{\ \ and\ \ } j\ \in
(\Z/p^r\Z)^\times,$$
where $\psi:=e^{\frac{2\pi i}{d}}$   and $\zeta_r:=e^{\frac{2\pi
i}{p^r}}$\  
(observe that the two previous definitions    
$\zeta:=e^{\frac{2\pi i}{p}}$ and $\zeta_r:=e^{\frac{2\pi i}{p^r}}$
imply that $\zeta_1=\zeta$).

\smallskip
 Let us  fix the  root of unity   $\xi :=\psi \,\zeta_r$.  
 \smallskip
Let  $L:= \Q(\xi)$ and $M=LK = \Q(\xi,\zeta)$.
Put ${\mathfrak q}  = (q, u\,\xi - v)$ where $\mk q$ is a prime ideal of
$L$ over $q$  because we have assumed $q\not|\  n$.
\footnote{Observe that the  prime ideal $\mk q$   
is fixed unambiguously by this  choice of $\xi$. }
We denote by ${\mathfrak Q} $ any prime ideal of $M$ over 
 ${\mathfrak q}$ and by $\mk q_K$ the prime ideal of $K$
 under  $\mk Q$. 

\smallskip
(i)  If $r=0$ then $L= \Q(\psi)$ and $M$ is of degree $p-1$ over
 $L$. 

\smallskip
(ii)  If $r \geq 1$ then $M=L$ and 
${\mathfrak Q} ={\mathfrak q}$.
\end{notas}

\smallskip
  Let us recall the definition of the $p$th power 
residue symbols in $K$ and $M$  with values in  $\mu_p$ (see [GQ]
definition 2.13).

\begin{defi}\label{d1103061}
If $\alpha \in M$ is prime to
${\mathfrak Q}\,\vert\, {\mathfrak q}$ in $M$, then let $\ov\alpha$ 
be the image of $\alpha$ in
the residue field $\Z_M/{\mathfrak Q} \simeq \mf F_{q^f}$; 
since
$\zeta \in \Z_M$, the image $\ov \zeta$ of $\zeta$ is of order $p$
(since $\zeta \not\equiv 1 \mod {\mathfrak Q}$) and we can put
$\ov\alpha_{}^{\,\kappa} = \ov \zeta^{\,\mu}$, $\kappa=\frac{q^f-1}{p}$,
$\mu \in \Z/p\Z$,
which defines the $p$th power residue symbol
$\big(\frac{\alpha}{{\mathfrak Q}}\big)_{\!\!M} :=
\zeta^\mu$; this symbol is equal to 1 if and only if $\alpha$ is a
local $p$th power at ${\mathfrak Q}$ (see
[Gr1,\,I.3.2.1, Ex.\,1]).

\smallskip\noindent
With this definition, for any automorphism $\tau\in {\rm Gal}(M/\Q)$ one
obtains, from $\alpha_{}^{\,\kappa} \equiv \zeta^{\,\mu}
\mod {\mathfrak Q}$,
$\tau \alpha_{}^{\,\kappa}  \equiv \tau \zeta^\mu
\mod \tau {\mathfrak Q}$, thus
$\tau \big(\frac{\alpha}{{\mathfrak Q}}\big)_{\!\!M} =
\big(\frac{\tau \alpha}{\tau {\mathfrak Q}}\big)_{\!\!M} =
\tau \zeta^\mu$.

\smallskip\noindent
If $\alpha \in K$, since ${\mathfrak q}_K\,\vert\,q$ in $K$ splits
totally in $M/K$, we have $\Z_K/{\mathfrak q}_K \simeq \Z_M/{\mathfrak Q}$
and
$\big(\frac{\alpha}{{\mathfrak q}_K}\big)_{\!\!K} =
\big(\frac{\alpha}{{\mathfrak Q}}\big)_{\!\!M}$
for any ${\mathfrak Q}\,\vert\,{\mathfrak q}_K$.
In particular this implies $\big(\frac{\zeta}{{\mathfrak 
q}_K}\big)_{\!\!K} = \zeta^{\kappa}$
(the symbol of $\zeta$ does not depend on the choice of ${\mathfrak 
q}_K \,\vert\, q$).
\end{defi}

\subsection {The main results}\label{s1111232} $ $

In the classical approach, the most part of the  results on FLT for the exponent $p$ 
are  obtained  by  localization at $p$ (Kummer, Mirimanoff, Wieferich, Vandiver and others)  
or  by some properties  of the $p$-class group of $K$ (Eichler).
There are less investigations with localizations at primes $q\not=p$
(Sophie Germain, Vandiver, Wendt,  Furtw\"angler, Krasner, D\'enes and others).

Revisiting  some  ideas of Vandiver for $FLT$ in [Va1, Va2]  involving 
a systematic use of   the $p$th  power residue symbols 
$\big(\frac{a}{\mk Q}\big)_{\!\! M}$ for  $a\in M$ 
coprime with $\mk Q$ (see definition \ref{d1103061}), this article is a  
 complement to the article [GQ] for localizations at primes $q\not=p$.

The main results of this article are :
Assume  that $SFLT2$ fails for $(p,u,v)$.
Let  $q$ be any  odd prime coprime with $puv$, $f$ the order of $q\mod p$,   $n$ the order of 
$\frac{v}{u}\mod q$, $\xi$ a primitive
$n$th
root of unity, $\mk q$ the prime ideal  $(q,u\xi-v)$ of $\Q(\xi)$.
\bn
\item
if  $q$ is {\it $p$-principal} and $n\not=2p$  we prove  (theorem \ref{t1304041}) that   
$$\Big(\frac{1+\xi\zeta^k}{1+\xi\zeta}\Big)^{(q^f-1)/p}\equiv 1\mod \mk q\mbox{\ for\ }k=1,\dots,p-1.$$
We shall  derive,  by example,  of this congruence that:
  
- With a probabilistic estimate, more  than half  the primes $r<p^{p/5}$ of even degree $\mod p$ should  divide $v$ 
(remark \ref{r1305241}).  

- If $q$ of order $f\mod p$  divides $(u^p+v^p)$ then $(1-\zeta)^{(q^f-1)/p}\equiv p^{-(q^f-1)/p}\mod q, $
(corollary \ref{c2} and \ref{c4}).

- If $q$ of order $f\mod p$  divides $(u^p-v^p)$ then $(1-\zeta)^{(q^f-1)/p}\equiv 1 \mod q, $
(corollary \ref{c2} and \ref{c3}).
These two last results  reinforce strongly the  first and second theorem of Furtw\"angler.

\item
If  $u+\zeta v\not\in K^{\times p}$, then $p$ is irregular and there exists an effective constant $C(p)$,
depending only on $p$ and  smaller than Minkowski Bound,    such that
there exists at least one  prime  $q<C(p)$  satisfying $q$ not dividing $uv$, $q$ non $p$-principal and 
$$\Big (\frac{\zeta^{-k^m}(1+\xi \zeta^k)}{\zeta^{-1}(1+\xi\zeta)}\Big)^{\frac{q^f-1}{p}}\equiv 1\mod \mk q\mbox{\ for\ }k=1,\dots,p-2,$$
for a certain integer $m\not\equiv 0\mod p$ and depending on $q$ (theorem \ref{t1012211}). 
\en 
The principle of proofs of the article  relies mainly  on the $p$-Hilbert class field theory.
Let us mention  that, in this complement, we limited ourselves  to the   second case of $SFLT$ and  principally to $p$-principal prime $q$.  
By opposite, we  took also in account the case where $p$ divides the order $n$ of $\frac{v}{u}\mod q,$  not examined in [GQ]. 
See also [Qu2] for some   improvements of these results in the second case $FLT2$ of Fermat's Last Theorem.

\section{The main theorem}
At first, we give  a definition and  an elementary lemma independent of $SFLT$.

\begin{defi}\label{d11O9221}
Let $n=dp^r$, with $d,p$ coprime and $r\geq 0$.
Let $\xi$ be a fixed primitive $n$th  root of unity
$\xi=\psi\zeta_r$ where  $\psi:=e^{\frac{2\pi i}{d}}$ and
$\zeta_r:=e^{\frac{2\pi i}{p^r}}$.

For all $\ 0\leq k <p-1$, let us define\,\footnote{The reason why  
$k = p-1$ is   discarded 
will be explained  in remark \ref{r1p1} after the lemma \ref{l2}.}
$$\varepsilon_k:= 1+\xi\zeta^k. \footnote{$\varepsilon_k$ is used with
this meaning
in the sequel of the article.}$$
\end{defi}

\begin{lem}\label{l1}$ $

a) If $k=0$, $\varepsilon_0 = 1 + \xi$  is a cyclotomic unit of $L$ 
except if $d=1$ ($\varepsilon_0 = 2$) or $d=2$  ($\varepsilon_0 = 0$).

\smallskip    
b) Suppose that $0<k<p-1$.

\smallskip

(i) If $d>2$ then  $\varepsilon_k=1+\xi\zeta^k$ is a cyclotomic unit.

\smallskip

(ii) If $d=2$ then $\varepsilon_k$ is not a cyclotomic unit and 

\smallskip
\hspace{0.6cm}--  If $r \geq 1$ then
$\varepsilon_k=1-\zeta_r^{1+kp^{r-1}}\in\Z[\zeta_r]$
with $\varepsilon_k\Z[\zeta_r]=\mk p_r$ 
where $\mk p_r$ is the prime ideal of $\Z[\zeta_r]$ above $p$.

\smallskip
\hspace{0.6cm}--  If $r=0$ then $\varepsilon_k= 1-\zeta^{k}$ with
$\varepsilon_k\Z_K=\mk p$.

\smallskip

(iii) If $d=1$ then $\varepsilon_k$ is a cyclotomic unit and 

\smallskip
\hspace{0.6cm}-- If $r \geq 1$ then $\varepsilon_k=
1+\zeta_r^{1+kp^{r-1}}$.

\smallskip
\hspace{0.6cm}--  If $r=0$ then $\varepsilon_k= 1+\zeta^{k}$.
\begin{proof}
Left to the reader.
\end{proof}
\end{lem}

\smallskip

The following lemma using Hilbert class field theory for  $K$ plays a
central role in the article.

\smallskip

\begin{lem}\label{l2}
Suppose that $SFLT2$ fails for  $(p,u,v)$. 
Let  $q\not|\ puv$ be a $p$-principal prime,  
$n$ the order of $\frac{v}{u}\mod q$, $\xi:=e^{\frac{2\pi i}{n}}$  and 
$\mk q$   the prime ideal $(u\xi-v, q)$ of  $\Z_L$.
Then 
$$\Big(\frac{1+\xi \zeta^k}{\mk Q}\Big)_{\!\! M}= \Big(\frac{u}{\mk
q_K}\Big)_{\!\! K}^{\!\!-1}
\mbox{\ for all\ }k=1,\ldots,p-2 \mbox{\  and all \ }\mk Q|\mk q \mbox{\ with \ } \mk q_K\mbox{\ under\ } \mk Q.
\footnote{Observe that $\big(\frac{u}{\mk q_K}\big)_{\!\!  K}$ does
not depend on $k$ 
and recall that  $\mk Q=\mk q$ if $r>0$.}$$  
\begin{proof}
Let us choose one $\mk Q$ over $\mk q$ with $\mk q_K$ under $\mk Q$ and observe what happens when 
$k$ varies thorough $1,\dots,p-1$:

- We have $u\xi- v\equiv 0 \mod \mk q$, so $u\xi - v\equiv 0\mod  \mk Q$,  hence 
with   $\gamma := u+\zeta  v$, we get
$s_k(\gamma)=u+\zeta^k v\equiv u(1+\xi\zeta^k)\equiv u \varepsilon_k\mod  
\mk Q$ for $k=1,\dots,p-1$.

- We obtain 
$\big(\frac{s_k(\gamma)}{\mk Q}\big)_{\!\! M}= 
\big(\frac{u}{\mk Q}\big)_{\!\! M}
\big(\frac{\varepsilon_k}{\mk Q}\big)_{\!\! M},$
so 
$\big(\frac{s_k(\gamma)}{\mk q_K}\big)_{\!\! K}= 
\big(\frac{u}{\mk q_K}\big)_{\!\! K}
\big(\frac{\varepsilon_k}{\mk Q}\big)_{\!\! M}.$

-  The numbers $s_k(\gamma)$ are $p$-primary pseudo-units, which implies 
 $\big(\frac{s_k(\gamma)}{\mk q_K}\big)_{\!\! K}=1$  for all $k\not\equiv
0\mod  p$ 
from the  decomposition of primes in Hilbert class fields
 because, by assumption,  $q$ is
$p$-principal.

- We can reiterate the same reasoning for all $\mk Q$ over $\mk q$.

\end{proof}
\end{lem}

\begin{rema}\label{r1p1}
{\rm We explain why we can discard the value $k=p-1$ of the index $k$.
The value $k=p-1$ is excluded  because $\varepsilon_k$ would
be null if and only if  $d=2,r=1$ and $k=p-1$.
This particular case shall be directly examined in the corollary \ref{c4}.
For all the other $(d,r,k),\ 1\leq k\leq p-1$ with  $\varepsilon_k\not=0$,
we show  now
that it is always possible to
express $\big(\frac{1+\xi\zeta^{p-1}}{\mk Q}\big)_{\!\! M}$ in function of
$\big(\frac{1+\xi\zeta^{k}}{\mk Q}\big)_{\!\! M},\ k=1,\dots,p-2$:
we start from $$u(1+\xi\zeta^j)\equiv s_j(\gamma)\mod\mk q,\mbox{\ for\ }
j=1,\dots,p-1,$$
where $1+\xi\zeta^j$ is always
nonzero, therefore 
$$u^{p-1}\prod_{j=1}^{p-1} (1+\xi\zeta^j)\equiv N_{K/\Q}(\gamma)=w_1^p\mod
\mk q,$$ 
so 
$$\Big(\frac{u^{p-1}(1+\xi\zeta)\ldots(1+\xi
\zeta^{p-2})(1+\xi\zeta^{p-1})}{\mk Q}\Big)_{\!\! M}=1
\ \mbox{\ for all\ } \mk Q|\mk q$$
so, from lemma \ref{l2} applied for all $1\leq k\leq p-2$ we get
$\Big(\frac{u(1+\xi\zeta^{p-1})}{\mk Q}\Big)_{\!\! M}=1,$
and thus $$\Big(\frac{1+\xi\zeta^{p-1}}{\mk Q}\Big)_{\!\!
M}=\Big(\frac{1+\xi\zeta^k}{\mk Q}\Big)_{\!\! M},
\ \mbox{\ for all\ } k=1,\dots,p-2.$$}
\end{rema}

\subsection {The general case}

Our main results on $SFLT2$ are a direct consequence of the lemma \ref{l2}  dealing with the  
decomposition of $\mk Q$ in a certain Kummer $p$-extension defined from 
the Vandiver's cyclotomic units. They will follow from:

\begin{thm}\label{t1304041}
Assume  that $SFLT2$ fails for  $(p,u,v)$. Let $q
\not|\ puv$  be a $p$-principal
prime, $n$ the order of  $\frac{v}{u}\mod q$, $\xi:=e^{\frac{2\pi
i}{n}}$ and $M:=\Q(\xi,\zeta)$.
For all $\ 0\leq k <p-1$,  let 
$\varepsilon_k:= 1+\xi\zeta^k$ and      ${\mathfrak q}$ be the prime ideal $(q, u\,\xi - v)$ of $\Z[\xi]$ over  $q$.

Then all the prime ideals   $\mk Q$ of $\Z[\xi,\zeta]$ over $\mk q$
split  totally 
\footnote{Recall that $\mk Q=\mk q$ if $p\ |\ n$.} 
in   the Kummer extension 
$$M \Big(\sqrt[p]{<\varepsilon_k\varepsilon_1^{-1}>_{k=1,\ldots,p-2}}\,
\Big) \big/M.$$

\begin{proof} It is a reformulation of the previous lemma \ref{l2}
where  $\big(\frac{\varepsilon_k}{\mk Q}\big)_{\!\! M}=
\big(\frac{\varepsilon_1}{\mk Q}\big)_{\!\! M}$, for $k=1,\dots,p-2$.
\end{proof}
\end{thm}

\begin{cor}\label{c1305221}
Assume  that $SFLT2$ fails for  $(p,u,v)$. Let $q\not|\ puv$  be a $p$-principal
prime, $n$ the order of  $\frac{v}{u}\mod q$, $\xi:=e^{\frac{2\pi i}{n}}$ and      ${\mathfrak q}$ be the prime ideal $(q, u\,\xi - v)$ of $\Z[\xi]$ over  $q$.

If $n\not=2p$ then
\be\label{e1305231}
\Big(\frac{1+\xi\zeta^k}{1+\xi\zeta}\Big)^{\frac{q^f-1}{p}}\equiv 1\mod \mk q\mbox{\ for all\ }k=1,\dots,p-1.
\ee
\begin{proof} $ $
\bn 
\item  If $p\ |\ n$ then we have $\mk q=\mk Q$ and, from theorem \ref{t1304041},  we get  for the pair  $(\mk q_K, \mk Q)$ with
$\mk q_K$ under $\mk Q$  
$$\Big(\frac{1+\xi\zeta^k}{1+\xi\zeta}\Big)^{\frac{q^f-1}{p}}\equiv 1\mod \mk q\mbox{\ for all \ } k=1,\dots,p-1. 
\footnote{In fact, we can consider  all $k=1,\dots,p-1$ and not only $k=1,\dots,p-2$ from remark \ref{r1p1} except in the case where $n=2p$
which is examined  directly in corollary \ref{c4}.}$$
\item  If $p\! \not |\ n$ then for  all $k=1,\dots,p-1,$   we have 
$$\Big(\frac{1+\xi\zeta^k}{1+\xi\zeta}\Big)^{\frac{q^f-1}{p}}\equiv 1\mod \mk Q \mbox{\ for all\ }\mk Q\ |\ \mk q,$$ 
because  for each   triple $(\mk q_K, \mk q, \mk Q)$ with $\mk Q$ over $\mk q$ and $\mk q_K$ under $\mk Q$
we have $\big(\frac{u}{\mk q_K}\big)_{\!\! K}\big(\frac{1+\xi\zeta^k}{\mk Q}\big)_{\!\! M}
=\big(\frac{u}{\mk q_K}\big)_{\!\! K}\big(\frac{1+\xi\zeta}{\mk Q}\big)_{\!\! M}=1.$
\en
\end{proof}
\end{cor}

\begin{rema} $ $
{\rm 
\bn
\item
The relation (\ref{e1305231}) is equivalent to :  for all  prime ideals $\mk Q$ of $M=\Q(\xi,\zeta)$ over $\mk q$
then $\mk Q$ splits totally in the $p$-Kummer extension $M \big(\sqrt[p]{<(1+\xi\zeta^k)/(1+\xi\zeta)>_{k=1,\ldots,p-2}}\,\big) \big/M.$
\item
Observe that lemma \ref{l2} implies similarly that:
\bn
\item 
if $\big(\frac{u}{\mk q_K}\big)_{\!\! K}=1$ then 
all the prime ideals   $\mk Q$ of $\Z[\xi,\zeta]$ dividing $\mk q$
split totally 
in   the Kummer extension 
$M \big(\sqrt[p]{<(1+\xi\zeta^k>_{k=1,\ldots,p-2}}\,\big) \big/M.$
\item 
if $\big(\frac{u}{\mk q_K}\big)_{\!\! K}\not=1$ then 
all the prime ideals   $\mk Q$ of $\Z[\xi,\zeta]$ dividing $\mk q$
are  inert 
in   the Kummer extension 
$M \big(\sqrt[p]{<(1+\xi\zeta^k>_{k=1,\ldots,p-2}}\,\big) \big/M.$
\en 
\en }
\end{rema}

\begin{cor}\label{c1304041}
Assume  that Vandiver conjecture holds for $p$ and that  $SFLT2$ fails for  $(p,u,v)$. Let $q\not=p$
be a prime whose order $f \mod p$ is even and such that $p$ does not divide $\frac{ q^f-1}{p}$.
Let  $n$ be  the order of  $\frac{v}{u}\mod q$,  $\xi:=e^{\frac{2\pi i}{n}}$  and  $\mk q$ be  the  prime ideal    $(q, u\,\xi - v)$ of $\Z[\xi]$ over  $q$.
If $n\not=2p$ then, either 
\be\label{e1305232}
\Big(\frac{1+\xi\zeta^k}{1+\xi\zeta}\Big)^{\frac{q^f-1}{p}}\equiv 1\mod \mk q\mbox{\ for all\ }k=1,\dots,p-1,
\ee
or $q$ divides $v$.
\begin{proof}

From the first theorem of Furtw\"angler for SFLT, see [GQ] cor 2.15 (i) and the assumption $\frac{q^f-1}{p}\not\equiv 0$, we derive that
$q$ does not divide $u$. The order of $q\mod p$ is even, so
$\big(\frac{u}{\mk q_K}\big)_{\!\! K}=1$ and $\mk q_K$ is $p$-principal because Vandiver's conjecture holds for $p$.
  Then we  apply corollary  \ref{c1305221}.
\end{proof}
\end{cor}

\begin{rema}\label{r1305241} $ $

\smallskip

{\rm 
\bn
\item 
 If $SFLT2$ fails  for  $(p,u,v)$ with  $p|v$   and
if the $p$-principal prime    $q < p$, the probability is very small, for    the ideal $\mk q$
of $L$ over $q$, to   split totally  in the Kummer extension 
$M \Big(\sqrt[p]{<\varepsilon_k\varepsilon_1^{-1}>_{k=1,\ldots,p-1}}\,
\ \Big) \big/M$: let $p^\delta$ be the degree of this 
Kummer extension;  the probability estimate  that $\mk q$ split totally in
this extension satisfies  $$\mathcal P  \leq \frac{\phi(n)}{p^\delta}\leq \frac{\phi(q-1)}{p^\delta}).$$
\item
As a  consequence, if  $SFLT2$  failed for   $(p,u,v)$  with $p|v$
sufficiently large,  these probability estimates   suggest  that the integer $|v|$ should be
a 
{\it very  large  integer}
 with a very large number of
$p$-principal primes $q|v$ with $\kappa\not\equiv 0\mod p$.
As an example, apply cor. \ref{c1304041} for $p=491$. Assume that $SFLT2$ fails for $p=491$.
The  primes $q<p$ of even degree $\mod p$ with $\kappa\not\equiv 0\mod p$ are : 
$2,7,19,23,29,47,53,59,67,73,89$, $103,109,113,137,149,151,157,167,173,191,193,193,211,251,271,281,283,307,311$,
$313,317,337,347,353,359,367,
373,383,397,421,431,433,439,443,439,479,487$. 
For these primes, we have $\ml P<\frac{1}{p^\delta-1}$ is very small.

We see that there is a large number of primes dividing $v$, a fortiori if we consider 
all such  primes of  even order $\mod p$  for $q< p^\nu$  with for instance  $\nu <\delta-2$. It is reasonable  to  think  that for $p$ sufficiently large 
we have $\delta<\frac{p}{4}$,  so  we can take $\nu<\frac{p}{5}$ to assert that  
{\it more  than half of   all  the  primes $q< p^{\frac{p}{5}}$ of even order $\mod p$   should not satisfy (\ref{e1305232}) and so should  divide $v.$}
We have not found in the FLT literature that $x^p+y^p+z^p=0$ with $p|y$ should imply that $y$ must have a so  very large number 
of small prime factors.
This observation  
 could bring some tools for  another  diophantine tackling of $SFLT2$ and $FLT2$, for instance  in the spirit of Thue, see Ribenboim \cite{rib} (3B) p. 233
or Baker's linear form of logarithm or others.
It is possible to formulate a  corollary  with  an assumption independent of the values $(u,v)$ of any  possible counter-example which should imply   that $SFLT2$  holds for $p$:
\begin{cor}\label{c1305201} Let $p>3$ be a prime.
Assume that 
 there exist infinitely many primes $q$  such that, for $\xi_{q-1}:= e^{\frac{2\pi i}{q-1}}$ and  $L:=\Q(\xi_{q-1})$, 
there is no  prime ideal $\mk q$ of $L $ such that 
$$\Big(\frac{1+\xi_{q-1}\zeta^k}{1+\xi_{q-1}\zeta}\Big)^{\frac{q^f-1}{p}}\equiv 1\mod \mk q\mbox{ \ for all\ }k=1,\dots,p-1.$$
Then $SFLT2$ holds for $p$.
\begin{proof}
Suppose that $SFLT2$ fails for $(p,u,v)$.
There are infinitely many $q$, so there is at least one  $q$ such that $uv\not\equiv 0\mod q$.  
Let $n:=q-1$, then $u^n-v^n\equiv 0\mod q$.  Let $\xi:=e^{\frac{2\pi i}{n}}$ and $L=\Q(\xi)$. There exists a prime 
ideal $\mk q$ of $\Z_L$ such that $u\xi_{q-1}-v\equiv 0\mod \mk q$. 
Then, by corollary \ref{c1304041},  we should have
$\Big(\frac{1+\xi_{q-1}\zeta^k}{1+\xi_{q-1}\zeta}\Big)^{\frac{q^f-1}{p}}\equiv 1\mod \mk q\mbox{\ for all\ }k=1,\dots,p-2,$
contradicting the assumptions  made on $q$.
\end{proof}
\end{cor}
The existence of an infinity of primes $q$ satisfying assumptions of corollary \ref{c1305201} is an {\it open question}
because $\frac{\phi(q-1)}{p^{p-1}}\ri \infty$ when $q\ri\infty$.
\en}
\end{rema}

\subsection {The case  $n\in\{p,1,2p, 2\}$}\label{s1304081}

In  this  subsection,  we suppose that
$SFLT2$ 
fails for  $(p,u,v)$ and  
we apply the lemma \ref{l2} in  fixing  $n\in\{p,1,2p, 2\}$ to
derive some strong properties of all the $p$-principal primes $q$ dividing
$\Phi_n(u,v)$ for these values of $n$.
Observe that  we have  $M=K$ in all these  cases.

\subsubsection{The two cases $n= p$  and $n=1$.}
The reunion of these  two cases  allows us to investigate the properties
of 
all  the  $p$-principal primes $q$ dividing $u^p-v^p$.
\begin{cor}\label{c2} 
Case $n=p$: suppose that $SFLT2$ fails for  $(p,u,v)$. 
If    $q$ is  a  $p$-principal  prime dividing $\frac{u^p-v^p}{u-v}$  then 
$$q\equiv 1\mod  p^2\mbox{\ and \ }(1+\zeta)^{(q-1)/p}\equiv u^{(q-1)/p}\equiv v^{(q-1)/p}\equiv
2^{(q-1)/p}\equiv 1\mod  q.$$

\begin{proof}
Here $\xi=\zeta$, $u\zeta-v\equiv 0\mod \mk q$, 
$\varepsilon_k=1+\zeta^{k+1}$, $M=L=K$ and $\mk q=\mk Q=\mk q_K$, so  
$\big(\frac{1+\zeta^{k+1}}{\mk q_K}\big)_{\!\! K}=\big(\frac{u}{\mk
q_K}\big)_{\!\! K}^{-1}
\ \ \mbox{\ for all\ } k=1,\dots,p-2$ from lemma \ref{l2}.
It follows that  
$\big(\frac{1+\zeta^{2}}{\mk q_K}\big)_{\!\! K}= 
\big(\frac{1+\zeta^{-2}}{\mk q_K}\big)_{\!\! K}$,
which implies that   $\big(\frac{\zeta}{\mk q_K}\big)_{\!\! K}=1$,
thus $q\equiv 1\mod  p^2$, observing that $q\equiv 1\mod  p$. 

\smallskip
 From  $\big(\frac{\zeta}{\mk q_K}\big)_{\!\! K}=1$, we get 
$\big(\frac{1+\zeta}{\mk q_K}\big)_{\!\! K}= 
\big(\frac{1+\zeta^{p-1}}{\mk q_K}\big)_{\!\! K}$,
so $\big(\frac{1+\zeta^{j}}{\mk q_K}\big)_{\!\! K}=\big(\frac{u}{\mk
q_K}\big)_{\!\! K}^{-1}$,
 for all $j=1,\dots,p-1$.
Thus  
$\big(\frac{1+\zeta}{\mk q_K}\big)_{\!\! K}=\ldots = 
\big(\frac{1+\zeta^{p-1}}{\mk q_K}\big)_{\!\! K}$.
Therefore 
$\Big(\frac{N_{K/\Q}(1+\zeta)}{\mk q_K}\Big)_{\!\!}=
\big(\frac{u^{-1}}{\mk q_K}\big)_{\!\! K}^{p-1}=1$, 
so $\big(\frac{u}{\mk q_K}\big)_{\!\! K}=\big(\frac{v}{\mk q_K}\big)_{\!\!
K}=1,$
and gathering these results we get 
$$\Big(\frac{1+\zeta}{\mk q_K}\Big)_{\!\! K}=\ldots =
\Big(\frac{1+\zeta^{p-1}}{\mk q_K}\Big)_{\!\! K}=\Big(\frac{u}{\mk
q_K}\Big)_{\!\! K}=
\Big(\frac{v}{\mk q_K}\Big)_{\!\! K}=1.$$

\smallskip

From $u\zeta-v\equiv 0\mod \mk q_K$,  
we have $w_1^p=\frac{u^p+v^p}{u+v}\equiv \frac{2u^p}{u(1+\zeta)}\mod \mk
q_K$, 
so $\big(\frac{2}{\mk q_K}\big)_{\!\! K}=1$,
and  finally 
$$\Big(\frac{2}{\mk q_K}\Big)_{\!\! K}=\Big(\frac{1+\zeta}{\mk
q_K}\Big)_{\!\! K}=\dots =
\Big(\frac{1+\zeta^{p-1}}{\mk q_K}\Big)_{\!\! K}=\Big(\frac{v}{\mk
q_K}\Big)_{\!\! K}
=\Big(\frac{u}{\mk q_K}\Big)_{\!\! K}=1. $$ 

By  conjugation by $s_\ell$,  we get 
$\big(\frac{1+\zeta}{s_\ell(\mk q_K)}\big)_{\!\!K}=1$  for any
$\ell\not\equiv 0\mod  p$, thus  
$$(1+\zeta)^{(q-1)/p}\equiv 1\mod  q. $$
\end{proof}
\end{cor}

\begin{cor}\label{c3} 
Case $n=1$: suppose that $SFLT2$ fails for  $(p,u,v)$. 
If    $q$ is    a  $p$-principal prime  of order $f\mod  p$ and  
$q$ divides $u-v$
then we have:
$$q^f\equiv 1\mod  p^2 \mbox{\ and\ }
u^{(q^f-1)/p}\equiv v^{(q^f-1)/p}\equiv (1+\zeta)^{(q^f-1)/p}\equiv   2^{(q^f-1)/p}\equiv 1\mod  q.$$

\begin{proof} $ $ Here $\varepsilon_k=1+\zeta^k$, $M=K$ and $L=\Q$.
 The proof is very similar to the case $n=p$ corollary  \ref{c2} starting
here from 
 the  relation
\be\label{e1305241}
u+\zeta^j v\equiv u(1+\zeta^j)\mod  q,
\ee
for all $j\not\equiv 0\mod  p$
(instead of a congruence  $\mod \  \mk q_K$),  
 observing that 
the degree of $q\mod  p$ can be here greater than $1$.
We have 
$N_{K/\Q}(1+\zeta)=1$
which implies that $\big(\frac{u}{\mk q_K}\big)_{\!\! K}=1$
 and  then  $\frac{u^p+v^p}{u+v}=w_1^p\equiv \frac{2^pu^p}{2u}\mod q$ 
implies $\big(\frac{2}{\mk q_K}\big)_{\!\!K}=1$.
\end{proof}
\end{cor}

\begin{rema} \label{r1011113}
{\rm 
Corollaries \ref{c2} and   \ref{c3}  
 imply   that all the $p$-principal primes $q$ dividing $u^p-v^p$   
 satisfy  $q^f\equiv 1 \mod  p^2$, which   brings a new generalization of
 the second Furtw\"angler's  theorem in    $SFLT2$ context 
obtained for the only   primes $q$ dividing $u-v$  in [GQ] cor. 2.16.
In the other hand,  the fact that $(1+\zeta)^{\frac{q^f-1}{p}}\equiv 1\mod q$  for all $q\ |\ u^p-v^p$ is new.
}
\end{rema}

\subsubsection{The two cases $n=2p$ and $n=2$}
The reunion of these two   corollaries of the theorem \ref{t1304041}
allows  us to investigate all the  $p$-principal primes $q$ 
dividing $u^p+v^p$ at the {\it core} of the $SFLT2$ equation.
We need to modify slightly the method to take into account the only values
$k$ with 
$\mk q_K$  co-prime with $u+\zeta^ k v$.

\begin{cor}\label{c4} 
Case $n=2p$ :  suppose that $SFLT2$ fails for  $(p,u,v)$. 
If $q$ is    a  $p$-principal prime  dividing $\frac{u^p+v^p}{u+v}$ 
then
$$ q\equiv 1\mod  p^2 \mbox{\ and\ } (1-\zeta)^{\frac{q-1}{p}}\equiv p^{-\frac{q-1}{p}}\equiv u^{-\frac{q-1}{p}}
\equiv v^{-\frac{q-1}{p}}\mod q.$$
\begin{proof}
Here,  we have $M=K=L$  and $\xi = -\zeta$ which implies that  
$v\equiv -\zeta u\mod \mk q=\mk q_K$,  thus
$$s_k(u+\zeta v)= u+\zeta^{k} v
=s_k(\gamma)\equiv u(1-\zeta^{k+1})\mod \mk q,\ \,  k=1,\dots p-2.$$
We obtain
$\big(\frac{u}{\mk q_K}\big)_{\!\! K}\big(\frac{1-\zeta^{k+1}}{\mk q_K}\big)_{\!\! K}=1$,
 for  $k\not\equiv  p-1\mod  p$ from lemma \ref{l2},
therefore $\big(\frac{1-\zeta^2}{\mk q_K}\big)_{\!\! K}= 
\big(\frac{1-\zeta^{p-2}}{\mk q_K}\big)_{\!\! K}$,
so $\big(\frac{\zeta}{\mk q_K}\big)_{\!\! K}=1$ and $q\equiv 1\mod  p^2$,
which implies that  
$\big(\frac{1-\zeta}{\mk q_K}\big)_{\!\! K} 
=\big(\frac{1-\zeta^{p-1}}{\mk q_K}\big)_{\!\! K}$. 
Gathering these results, we get
$$\Big(\frac{1-\zeta}{\mk q_K}\Big)_{\!\! K}=\dots 
=\Big(\frac{1-\zeta^{p-1}}{\mk q_K}\Big)_{\!\! K},$$
by multiplication we get 
$\big(\frac{u^{p-1}}{\mk q_K}\big)_{\!\! K}\big(\frac{p}{\mk 
q_K}\big)_{\!\! K}=1$ 
and finally:
$$\Big(\frac{u}{\mk q_K}\Big)_{\!\! K}= \Big(\frac{v}{\mk q_K}\Big)_{\!\!
K}=
\Big(\frac{p}{\mk q_K}\Big)_{\!\! K}
= \Big(\frac{1-\zeta^j}{\mk q_K}\Big)_{\!\! K}^{\!\!-1}, \ \mbox {\ for 
all\ }  j\not\equiv 0\mod  p. $$ 
We get $\big(\frac{(1-\zeta^j) p}{\mk q_K}\big)_{\!\! K}=1$ for all $j\not\equiv 0\mod p$, so  $((1-\zeta)p )^{\frac{q-1}{p}}\equiv 1\mod q$
and finally $ (1-\zeta)^{(q-1)/p}\equiv p^{-\frac{q-1}{p}}\mod q.$
\end{proof}
\end{cor}

\begin{cor}\label{c6}
Case $n=2$:  
suppose that $SFLT2$ fails for  $(p,u,v)$.
If  $q$ is  a  $p$-principal prime  dividing $u+v$ of degree $f\mod  p$
 then 
$$ q^f\equiv 1\mod  p^2 \mbox{\ and\ } (1-\zeta)^{\frac{q^f-1}{p}}\equiv p^{-\frac{q^f-1}{p}}
\equiv u^{-\frac{q^f-1}{p}}\mod q\equiv v^{-\frac{q^f-1}{p}}\mod q.$$

\begin{proof}
  Here, $\varepsilon_j=1-\zeta^j$ for $j\not\equiv 0\mod  p$, $M=K$ and
$L=\Q$.
In that case,  we get 
$\big(\frac{u}{\mk q_K}\big)_{\!\! K}\big(\frac{1-\zeta^j}{\mk q_K}\big)_{\!\! K}=1 \ \mbox{\ for all\ }  j\not\equiv 0\mod  p $.
The end of the proof is similar to that of corollary  \ref{c4}.
\end{proof}
\end{cor}

\begin{rema} \label{r1011113}
{\rm Corollaries  \ref{c4} and  \ref{c6}  
 imply   that all the $p$-principal primes $q$ dividing $u^p+v^p$   
 satisfy $q^f\equiv 1 \mod  p^2$, which   brings a generalization of
 the first  Furtw\"angler's  theorem in   the $SFLT2$ 
context  obtained for the only   primes $q$ dividing $u+v$ in [GQ] cor. 2.15.
In the other hand, the  fact that $(1-\zeta)^{\frac{q^f-1}{p}}\equiv p^{-\frac{q^f-1}{p}}\mod q$ for the primes dividing $u^p+v^p$ is new.}
\end{rema}

\subsection{The  case  $u+\zeta v\not\in K^{\times p}$}\label{snp}$ $

In this subsection,  we  assume that $SFLT2$ fails for $(p,u,v)$ with $p\ |\ v$ and  $u+\zeta v\not\in K^{\times p}$.
It follows that $p$ is irregular, if not the $p$-primary pseudo-unit  $u+\zeta v$ should belong to $K^{\times p}$, see for instance 
Gras [Gr2] thm 2.2.

\begin{lem}\label{l1108201}
Let $q\not =p$ be a prime number  and
${\mk q_K}$ any prime ideal of $\Z_K$ over  $q$.
 If   $q$ divides $u$ (resp $v$)  then    $\big(\frac{v}{\mk q_K}\big)_{\!\!K}=1$  (resp.  $\big(\frac{u}{\mk q_K}\big)_{\!\!K}=1$).
\end{lem}

\begin{proof}
We have $N_{K/\Q} (u+ v\zeta )=\frac{u^p+v^p}{u+v}=w_1^p$, 
where $N_{K/\Q} ({\mk w}_1)=w_1$;
so $q\ |\ v$ implies that
$u^{p-1}\equiv w_1^p\pmod q$, which leads to 
$\big(\frac{u}{\mk q_K}\big)_{\!\!K}=1$. Similar proof starting from
$q\ |\ u$.
\end{proof}

Let $\mathcal S$ be  a   finite set of   non $p$-principal
primes $q$  
such that the set of $p$-classes $c\ell(\mk q_K)\in C\ell$
 of  the prime ideals $\mk q_K$ of $K$ over  $q$  
generates the  $p$-elementary $p$-class group 
$C\ell_{[p]}$ of $K$. Let  us note $ Q_p$
 the greatest  prime
$q\in\mathcal S$. 
Let the Minkowski  Bound of $K$ given by 
$$\mathcal
B_p:=\Big(\frac{4}{\pi} \Big)^{(p-1)/2}\frac{(p-1)!}{(p-1)^{p-1}}\sqrt{p^{p-2}}.$$
With these definitions, we can always choose  a set $\ml S$   such that  $Q_p$ be  smallest possible with  $Q_p\leq \mathcal B_p$. 
Under the 
General Riemann Hypothesis GRH, we  know  that the whole ideal class group
of $K$  is generated by the  set of prime 
ideals ${\mk l}$ with
\be\label{e1106091}
N_{K/\Q}({\mk l}) < B:= 12 ({\rm log}\ \Delta_K)^2,
\ee
where $\Delta_K$ is the absolute discriminant of $K$
(see   [BDF]). Under GRH, we have generally  $Q_p\ll\mathcal
B_p$ (where  here $\ll$ means {\it very small compare to}) as soon as $p$ is large.

\medskip

\begin{lem}\label{l1106092}
Suppose that $u+\zeta v \notin K^{\times p}$.
Then there exists at least one  prime
$q \in \mathcal S$  such that  $uv\not\equiv 0\pmod q$.
\end{lem}

\begin{proof}
Let $\gamma$  be a $p$th root of $u+\zeta v$,     
$\gamma:=\sqrt[p]{u+\zeta v}$.
Let $H_1$ be the $p$-elementary Hilbert class field of $K$ (so that
${\rm Gal} (H_1/K) \simeq C\ell_{[p]}$). Let $N_1$ be a subextension of
$H_1$ such that $H_1$ is the direct compositum of $N_1$
and $K(\sqrt [p] \gamma)$ over $K$.

Therefore there exists at least one   prime $q\in \mathcal S$ such that  
the Frobenius of all  the prime ideals ${\mk q}_K$ over $q$ in $H_1/K$ are  of order $p$ and fix
$N_1$, so that their  restriction to $K(\sqrt [p] \gamma)/K$ are  of order $p$.
Thus
$\big(\frac{u+\zeta v}{{\mk q}_K} \big)_{\!\!K}\not =1$.

(i) If $q \ |\  v$,  we get a  contradiction with  lemma \ref{l1108201},
so $v\not\equiv 0\mod q$.

(ii) If $q \ | \ u$ we have $\big(\frac{u+\zeta v}{{\mk
q}_K} \big)_{\!\!K} =\big(\frac{\zeta v}{{\mk
q}_K} \big)_{\!\!K}=\big(\frac{\zeta }{{\mk q}_K }\big)_{\!\!K}$
because $\big(\frac{ v}{{\mk q}_K} \big)_{\!\!K}=1$ from  Lemma
\ref{l1108201},
thus  $\big(\frac{u+\zeta v}{{\mk q}_K} \big)_{\!\!K}=1$
since $\kappa\equiv 0\pmod p$ from the first Furtwangler's theorem for
SFLT (see
[GQ, Corollary 2.15]), which  brings  also a contradiction with
$\big(\frac{u+\zeta
v}{{\mk q}_K} \big)_{\!\!K}\not =1$. Therefore $u\not\equiv 0\pmod q$.
\end{proof}

\begin{defi}
{\rm 

\smallskip

For a definition of the  character of Teichm\"uller $\omega$ 
of $Gal(K/\Q)$, see for instance [GQ] definition 2.8. 
Let us consider the characters  $\chi_i=\omega^i, \ 1\leq i\leq p-1$. 
Let $\mathcal E$ be the group of $p$-primary pseudo-units  of $K$ seen as a $\F_p[g]$-module,  
and the $\chi_i$-components $\ml E_i:=\mathcal E^{e_{\chi_{i}}}$  of $\mathcal E$. 
The components  $\mathcal E_i$ are not all
trivial because $p$ is irregular.
}
\end{defi}

\begin{thm}\label{t1012211}
Suppose that $SFLT2$ fails for  $(p,u,v)$   with  $u+\zeta v\not\in
K^{\times p}$.  Then
$p$ is irregular and there exists at least 
one non $p$-principal prime $q\in \mathcal S$   such that: 
\bn
\item
We have $q\not|\  uv. $ 

\item 
$n$ being  the order   of $\frac{v}{u}\mod q$,  $\xi:=e^{\frac{2\pi
i}{n}}$, $\mk q$ the 
 prime ideal $(u\xi- v, q)$  of $\Z[\xi]$ over $q$, we have  
$$\Big (\frac{\zeta^{-k^m}(1+\xi \zeta^k)}{\zeta^{-1}(1+\xi\zeta)}\Big)^{\frac{q^f-1}{p}}\equiv 1\mod \mk q
\mbox{\ for\ }k=1,\dots,p-2,$$
for a certain  $m\in (\Z/p\Z)^\times$.
\end{enumerate}
\end{thm}

\begin{proof}$ $
\bn
\item
$p$ is irregular as seen above. From lemma \ref{l1106092} it is possible to choose 
$q\in \mathcal S$  with $uv\not\equiv 0\mod p$. 
\smallskip 
\item
The pseudo-unit $\gamma=\sqrt[p]{u+\zeta v}$  is not a $p$th power, hence in the decomposition
$\gamma=\prod_{{\chi_i}}\gamma^{e_{\chi_i}}$
on the p-1 characters $\chi_i,\  i=1,\dots,p-1$, there exists at least one  $i=m$ such that the pseudo-unit  $\gamma^{e_{\chi_m}}$ be  not a $p$-power.
Let us name   $\gamma_m$ this idempotent.

$\gamma_m$ is a $p$-primary pseudo-unit;
from Hilbert's class field theory
and lemma \ref{l1106092} applied with $H_1$ and $\gamma_m$, it is possible to choose  one $\mk q_K\in \ml S $  
such that 
$$\Big(\frac{\gamma_m}{\mk q_{K}}\Big)_{\!\!K}=\zeta^{w_m}\mbox{\ with\ } w_m\not\equiv 0\mod p
\mbox{\ and\ }\Big(\frac{\gamma_i}{\mk q_{K}}\Big)_{\!\!K}=1 \mbox{\ for all \ } i\not=m.$$
\item
Here the extension $M(\sqrt[p]{\gamma_m})$ is  Galois on $\Q$  because its  Galois group 
 acts in letting  globally  unvarying the radical, when raising to a power prime to $p$,   by use of the idempotent.
We can always change $\mk q_K$ in acting  by  conjugation to obtain $w_m=1$ and so 
$$\Big(\frac{\gamma}{\mk q_K}\Big)_{\!\!K}= \Big(\frac{\gamma_m}{\mk q_K}\Big)_{\!\!K}=\zeta.$$ 
\item
From $s_k(\gamma)=u+\zeta^k v$ for $k=1,\dots,p-2$ and $u\xi-v\equiv
0\mod\mk q$ we get $$s_k(\gamma)\equiv u\varepsilon_k\mod \mk q,$$ 
where $\varepsilon_k=1+\xi\zeta^k$,
so,  as in lemma \ref{l2},  $s_k(\gamma)\equiv u\varepsilon_k\mod \mk Q$ for all $\mk Q$ over $\mk q$ and 
$$\Big(\frac{s_k(\gamma)}{\mk q_K}\Big)_{\!\!K}
=\Big(\frac{u\varepsilon_k}{\mk Q}\Big)_{\!\!M}=\Big(\frac{s_k(\gamma_m)}{\mk q_K}\Big)_{\!\!K}=\zeta^{ k^m}$$
because $\gamma_m$ is an idempotent,
so 
$$\Big(\frac{u(1+\xi\zeta^k)}{\mk Q}\Big)_{\!\!M}=\zeta^{k^m},$$
and also  
$$\Big(\frac{u(1+\xi\zeta)}{\mk Q}\Big)_{\!\!M}=\zeta,$$
which leads to 
$$\Big (\frac{\zeta^{-k^m}(1+\xi \zeta^k)}{\zeta^{-1}(1+\xi\zeta)}\Big)^{\frac{q^f-1}{p}}\equiv 1\mod \mk q\mbox{\ for\ }k=1,\dots,p-1.$$
\en
\end{proof}

This theorem leads us to set the following {\it criterion} depending only on $p$
\footnote{We use intentionally    the term {\it criterion} to indicate
that  corollary \ref{c1012211} allows us (at least theoretically) in a
finite number of arithmetic computations
to  determine if,  for $p$ given,   the SFLT2 equation can be reduced to the
form $u+\zeta v\in K^{\times p}$.} 
to reduce the SFLT2 equation  to the  form $u+\zeta v\in K^{\times
p}$ for the irregular prime  $p$:
\begin{cor}\label{c1012211}
Let $p$ be an odd  irregular prime. Assume that SFLT2 fails for $p$ and that 
there is no  integer $1\leq m\leq p-1$, 
no prime $\mk q$ in $\Z[\xi_{q-1}]$ over a prime  $q\in \ml S$ with $\xi_{q-1}$ a $(q-1)$th primitive root of unity
such  that 
$$\Big (\frac{\zeta^{-k^m}(1+\xi_{q-1} \zeta^k)}{\zeta^{-1}(1+\xi_{q-1}\zeta)}\Big)^{\frac{q^f-1}{p}}\equiv 1\mod \mk q\mbox{\ for\ }k=1,\dots,p-2.$$
Then  the solution(s) of the $SFLT2$  equation take(s) the reduced form
$u+\zeta v\in K^{\times p}$.
\end{cor}

\begin{rema}
{\rm 
Suppose, as an example,  that $SFLT2$ fails for  $p$ 
and that  $p\|\mathcal C\ell_ K$  class group  of $K$, which implies that
$Card(\ml S)=1$.
Under GRH, from [BDF],  the definition of $\mathcal S$ should imply  
that $Q_p< 12(p-2)^2 (log\ p)^2$.
For one $q\in\mathcal S$, the probability estimate  $\mathcal P$ that $\mk
q$ split totally in  the Kummer extension
$$M 
\Big(\sqrt[p]{<((1+\xi\zeta^k)\zeta^{- k^m})/((1+\xi\zeta)\zeta^{-1})>_{k=1,\dots,p-2}}\, 
\ \Big)/M,$$
of degree $p^\delta$ are
$\mathcal  P\leq   \frac{ Q_p}{p^{\delta}}$,  
because $\phi(q)\leq  Q_p$ for
$q\in\mathcal S$.
Note that often,  and perhaps for all irregular primes $p>10^3$, we have
$\delta>\frac{p}{4}$ and, under GRH, we have   $Q_p<p^3$, so 
$$\mathcal P <\frac{1}{p^{\frac{p}{4}-3}}.$$
}
\end{rema}
\paragraph{ Acknowledgments:}
I would like to thank Georges   Gras and Preda Mihailescu   for pointing out  some  errors   
and  suggesting  some improvements 
for  the content and form of the article.

\end{document}